\documentclass[12pt]{amsart}

\usepackage{amsthm,amsmath,amsfonts,amssymb}
\usepackage{color}
\usepackage{tikz}
\usetikzlibrary{shapes,arrows,backgrounds,positioning,plotmarks,calc,patterns,plothandlers,decorations.pathmorphing,decorations.pathreplacing}
\usepackage{enumerate}
\usepackage{enumitem}
\usepackage{algorithmic}
\usepackage[T1]{fontenc}
\usepackage{wrapfig}
\usepackage{todonotes}

\usepackage[
            pdfauthor   = {Markus Lange-Hegermann},
            pdftitle    = {},
            pdfsubject  = {},
            pdfkeywords = {},
            bookmarks=true,
            bookmarksopen=true,
            colorlinks=true,
            hyperindex=true,
            linkcolor=blue,
            citecolor=blue,
            urlcolor=blue,
            ]{hyperref}

\newcommand{\mcI}{{\mathcal{I}}}

\newcommand{\Z}{{\mathbb{Z}}}
\newcommand{\Q}{{\mathbb{Q}}}

\newcommand{\C}{{\mathbb{C}}}
\DeclareMathOperator{\ld}{ld}
\DeclareMathOperator{\ini}{ini}
\DeclareMathOperator{\ord}{ord}
\DeclareMathOperator{\sep}{sep}
\DeclareMathOperator{\Ass}{Ass}
\DeclareMathOperator{\trdeg}{trdeg}
\DeclareMathOperator{\quot}{K}

\newtheorem{theorem}{Theorem}[section]
\newtheorem{lemma}[theorem]{Lemma}
\newtheorem{proposition}[theorem]{Proposition}

\theoremstyle{definition}

\newtheorem{example}[theorem]{Example}

\author{Markus Lange-Hegermann}
\address{Lehrstuhl B f\"ur Mathematik, RWTH Aachen University}
\email{\href{mailto:Markus Lange-Hegermann <markus.lange.hegermann@rwth-aachen.de>}{markus.lange.hegermann@rwth-aachen.de}}

\begin{document}

\title[Differential Dimension Polynomial]{The Differential Dimension Polynomial for Characterizable Differential Ideals}

\begin{abstract}
  We generalize the differential dimension polynomial from prime differential ideals to characterizable differential ideals.
  Its computation is algorithmic, its degree and leading coefficient remain differential birational invariants, and it decides equality of characterizable differential ideals contained in each other.
\end{abstract}

\subjclass[2010]{}

\maketitle

\section{Introduction}

Many systems of differential equations do not admit closed form solutions and hence cannot be solved symbolically.
Despite this, increasingly good heuristics to find solutions symbolically are implemented in computer algebra systems \cite{ChebTerrab1995102,ChebTerrabRoche2008}.
Given such a set of closed form solutions returned by a computer algebra system, the question remains whether this set is the complete solution set (cf. Example~\ref{example_heat_burgers}).

Classical measures like the Cartan characters \cite{Cartan1931} and Einstein's strength \cite{EinsteinRelativityAppenix2Supplement} describe the size of such solution sets.
However, they have a drawback: one can easily find two systems $S_1$ and $S_2$ of differential equations such that the solution set of $S_1$ is a proper subset of the solution set of $S_2$, but these two solution sets have identical measures (cf.\ Example~\ref{example_cartan_weak}).
In particular, these measures cannot detect whether a set returned by a solver of differential equations contains the difference of the solution sets of $S_2$ and $S_1$.

Kolchin introduced the differential dimension polynomial to solve this problem for solution sets of systems of differential equations corresponding to prime differential ideals \cite{KolPolynomial,JohnsonPolynomial,Kol,DifferentialDimensionPolynomials}.
This polynomial generalizes the Cartan characters and strength by counting the freely choosable power series coefficients of an analytical solution.
Recently, Levin generalized the differential dimension polynomial to describe certain subsets of the full solution set of a prime differential ideal \cite{LevinDimensionPolynomialIntermediateFields}.

Even though decomposing a set of differential equations into prime differential ideals is theoretically possible, it is expensive in practice (cf.\ \cite[\S6.2]{BLOP}).
Thus, there is a lack of practical methods which decide whether a subset of the solution set of differential equations is proper.
This paper solves this problem for greater generality than solution sets of prime differential ideals by generalizing the differential dimension polynomial to characterizable differential ideals.
Such ideals can be described by differential regular chains, and there exist reasonably fast algorithms that decompose a differential ideal into such ideals \cite{BLOP,thomasalg_jsc}.

To formulate the main theorem, we give some preliminary definitions; the missing definitions are given in Section~\ref{section_preliminaries}.
Denote by $F\{U\}$ a differential polynomial ring in $m$ differential indeterminates for $n$ commuting derivations over a differential field $F$ of characteristic zero.
For a differential ideal $I$ in $F\{U\}$ let $I_{\le \ell}:=I\cap F\{U\}_{\le \ell}$, where $F\{U\}_{\le \ell}$ is the subring of $F\{U\}$ of elements of order at most $\ell$.
We define the differential dimension function using the Krull dimension as
\begin{align*}
  \Omega_I:\Z_{\ge0}\mapsto\Z_{\ge0}:\ell\mapsto\dim(F\{U\}_{\le \ell}/I_{\le \ell})\mbox{ .}
\end{align*}
By the  following theorem, this function is ultimatively polynomial if $I$ is characterizable.
Such polynomials mapping $\Z$ to $\Z$ are called numerical polynomials, and there exists a natural total order $\le$ on them.

\begin{theorem}\label{theorem_general_dimension_polynomial}
  Let $I\subset F\{U\}$ be a characterizable differential ideal.
\begin{enumerate}
  \item There exists a numerical polynomial $\omega_I(\ell)\in\Q[\ell]$, called differential dimension polynomial, with $\omega_I(\ell)=\Omega_I(\ell)$ for sufficiently big $\ell\in\Z_{\ge 0}$.\label{theorem_general_dimension_polynomial_exists}
  \item $0\le\omega_I(\ell)\le m \binom{\ell+n}{n}$.
        In particular, $d_I:=\deg_\ell(\omega_I)\le n$.\label{theorem_general_dimension_polynomial_form}
  \item The degree $d_{I}$ and the coefficients $a_i$ for $i\ge d_I$ are invariant under differential birational maps, when writing $\omega_I(\ell)=\sum_{i=0}^na_i\binom{\ell+i}{i}$ with $a_i\in\Z$ for all $i\in\{0,\ldots,n\}$.\label{theorem_general_dimension_polynomial_invariant}
  \item The coefficient $a_n$ is the differential dimension of $F\{U\}/I$.\label{theorem_general_dimension_polynomial_differential_dimension}
\end{enumerate}
  Let $I\subseteq J\subset F\{U\}$ be another characterizable differential ideal.
\begin{enumerate}[resume]
  \item Then $\omega_J\le\omega_I$. \label{theorem_general_dimension_polynomial_two_ideals}
\end{enumerate}
  Assume $\omega_I=\omega_J$, and let $S$ and $S'$ be differential regular chains with respect to an orderly differential ranking $<$ that describe $I$ resp.\ $J$.
\begin{enumerate}[resume]
  \item The sets of leaders of $S$ and $S'$ coincide, and \label{theorem_general_dimension_polynomial_leaders}
  \item $I=J$ if and only if $\deg_x(S_x)=\deg_x(S_x')$ for all leaders $x$ of $S$, where $S_x$ is the unique element in $S$ of leader $x$. \label{theorem_general_dimension_polynomial_equality}
\end{enumerate}
\end{theorem}

This theorem can be slightly strengthened, as $I\subseteq J$ and $\omega_I=\omega_J$ already imply $\deg_x(S_x)\le\deg_x(S_x')$ for all leader $x$ of $S$ (cf.\ Lemma~\ref{lemma_pseudo_reduction}).
Thus $I=J$ if and only if $\prod_{x}\deg_x(S_x)=\prod_{x}\deg_x(S_x')$.
It would be interesting to have a version of this theorem, where this product is an intrinsic value, similar to the leading differential degree \cite{GaoChowForm}.

A more detailed description of the content of this paper in the language of simple systems is a part of the author's thesis \cite{lhphd}.

Section~\ref{section_proof_dimension_polynomial} proves Theorem~\ref{theorem_general_dimension_polynomial}, Section~\ref{section_computation} discusses the computation of the differential dimension polynomial, and Section~\ref{section_examples} gives examples.

\section{Preliminaries}\label{section_preliminaries}

\subsection{Squarefree Regular Chains}

Let $F$ be a field of characteristic zero, $\overline{F}$ its algebraic closure, and $R:=F[y_1, \ldots, y_n]$ a polynomial ring.
We fix the total order, called ranking, $y_1<y_2<\ldots<y_n$ on $\{y_1,\ldots,y_n\}$.
The $<$-greatest variable $\ld(p)$ occuring in $p\in R\setminus F$ is called leader of $p$.
The coefficient $\ini(p)$ of the highest power of $\ld(p)$ in $p$ is called initial of $p$.
We denote the separant $\frac{\partial p}{\partial\ld(p)}$ of $p$ by $\sep(p)$.

Let $S\subset R\setminus F$ be finite.
Define $\ld(S):=\{\ld(p)|p\in S\}$ and similarly $\ini(S)$ and $\sep(S)$.
The set $S$ is called triangular if $|\ld(S)|=|S|$; in this case denote by $S_x\in S$ the unique polynomial with $\ld(S_x)=x$ for $x\in\ld(S)$.
We call the ideal $\mcI(S):=\langle S\rangle:\ini(S)^\infty\subseteq R$ the ideal associated to $S$.
Let $S_{<x}:=\{p\in S|\ld(p)<x\}$ for each $x\in\{y_1,\ldots,y_n\}$.
The set $S$ is called squarefree regular chain if it is triangular and neither $\ini(S_x)$ nor $\sep(S_x)$ is a zero divisor modulo $\mcI(S_{<x})$ for each $x\in\ld(S)$.

\begin{proposition}[{\cite[Prop.~5.8]{Hubert1})}]\label{proposition_general_non_zerodivisor}
  Let $S$ be a squarefree regular chain in $R$ and $1\le i\le n$.
  Then $\mcI(S_{<y_i})\cap F[y_1,\ldots,y_{i-1}] = \mcI(S)\cap F[y_1,\ldots,y_{i-1}]$.
  Furthermore, if $p\in F[y_1,\ldots,y_{i-1}]$ is not a zero-divisor modulo $\mcI(S_{<y_i})$, then $p$ is not a zero-divisor modulo $\mcI(S)$.
\end{proposition}

\begin{theorem}[Lazard's lemma, {\cite[Thm.~4.4, Coro.~7.3, Thm.~7.5]{Hubert1}, \cite[Thm.~1]{BLOP}}]\label{theorem:krull_counting}
  Let $S$ be a squarefree regular chain in $R$.
  Then $\mcI(S)$ is a radical ideal in $R$, and the set $\{y_1,\ldots,y_n\}\setminus \ld(S)$ forms a transcendence basis for every associated prime of $\mcI(S)$.
  Let such an associated prime $\mcI(S')$ be given by a squarefree regular chain $S'$.
  Then $\ld(S)=\ld(S')$ and, in particular, $R/\mcI(S)$ is equidimensional of dimension $n-|S|$.
\end{theorem}

\subsection{Differential Algebra}

Let $F$ be a differential field of characteristic zero with pairwise commuting derivations $\Delta=\{\partial_1, \ldots, \partial_n\}$.
Let $U:=\{ u^{(1)}, \ldots, u^{(m)} \}$ be a set of differential indeterminates and define $u^{(j)}_\mu:=\partial^\mu u^{(j)}$ for $\partial^\mu := \partial^{\mu_1}_1\ldots\partial^{\mu_n}_n$, $\mu \in (\Z_{\ge 0})^n$.
For any set $S$ let $\{S\}_\Delta:=\{\partial^\mu s | s\in S, \mu \in (\Z_{\ge 0})^n\}$.
The differential polynomial ring $F\{U\}$ is the infinitely generated polynomial ring in the indeterminates $\{U\}_\Delta$.
The derivations $\partial_i: F \to F$ extend to $\partial_i: F\{U\} \to F\{U\}$ via additivity and Leibniz rule.
We denote the differential ideal generated by $p_1,\ldots,p_s\in F\{U\}$ by $\langle p_1,\ldots,p_s\rangle_\Delta$.

A ranking of the differential polynomial ring $F\{U\}$ is a total ordering $<$ on the set $\{U\}_\Delta$ satisfying additional properties (cf.\ e.g.\ \cite[p.\ 75]{Kol}).
A ranking $<$ is called orderly if $|\mu|<|\mu'|$ implies $u^{(j)}_{\mu} < u^{(j')}_{\mu'}$, where $|\mu|:=\mu_1+\ldots+\mu_n$.
In what follows, we fix an orderly ranking $<$ on $F\{U\}$.
The concepts of leader, initial and separant carry over to elements in the polynomial ring $F\{U\}$.

For a commutative ring $R$ denote the total quotient ring by $\quot(R)$.
Let $R$ and $R'$ be residue class rings of a differential polynomial ring by a differential ideal.
A differential birational map from $R$ to $R'$ is an isomorphism $\varphi: \quot(R) \to \quot(R')$ of $F$-algebras that commutes with derivations.
A differential transcendence basis $\{p_1,\ldots,p_d\}\subset R$ is a maximal set with $\biguplus_{i=1}^d\{p_i\}_\Delta$ algebraically independent over $F$.
The differential dimension of $R$ is the corresponding cardinality $d$.

A finite set $S\subset F\{U\}\setminus F$ is called (weakly) triangular if $\ld(p)$ is not a derivative of $\ld(p)$ for all $p,q\in S$, $p\not=q$.
Define $S_{<x}$ and $S_x$ as in the algebraic case.
We call $\mcI(S):=\langle S\rangle_\Delta:(\ini(S)\cup \sep(S))^\infty\subseteq F\{U\}$ the differential ideal associated to $S$.
The set $S$ is called coherent if the $\Delta$-polynomials of $S$ reduce to zero with respect to $S$ \cite{Rosenfeld}, and it is called differential regular chain if it is triangular, coherent, and if neither $\ini(S_x)$ nor $\sep(S_x)$ is a zero divisor modulo $\mcI(S_{<x})$ for each $x\in\ld(S)$.
An ideal $\mcI(S)$ is called characterizable if $S$ is a differential regular chain.

Let $S$ be a differential regular chain in $F\{U\}$, $\ell\in\Z_{\ge0}$, and $L:=\{\partial^\mu y|y\in \ld(S)\}\cap F\{U\}_{\le\ell}$ be the set of derivatives of leaders of elements in $S$ of order at most $\ell$.
For each $x\in L$ there exists a $\mu_{[x]}\in\Z_{\ge0}^n$ and a $p_{[x]}\in S$ such that $\ld(\partial^{\mu_{[x]}} p_{[x]})=x$.
Define an algebraic triangular set associated to $S$ as $S_{\le\ell}:=\{\partial^{\mu_{[x]}} p_{[x]}|x\in L\}$.
Although $S_{\le\ell}$ depends on the choice of $\mu_{[x]}$ and $p_{[x]}$, it has properties independent of the choice.

\begin{lemma}[Rosenfeld's lemma \cite{Rosenfeld}]\label{lemma_two_ways_to_get_algebraic_ideals}
  Let $S$ be a differential regular chain in $F\{U\}$, $\ell\in\Z_{\ge0}$, and $<$ orderly.
  Then $S_{\le\ell}$ is a squarefree regular chain and $\mcI_{F\{U\}_{\le \ell}}(S_{\le \ell})=\mcI(S)_{\le \ell}$.
\end{lemma}

\subsection{Numerical Polynomials}\label{subsection_numerical_polynomials}

Numerical polynomials are elements in the free $\Z$-module $\left\{\binom{\ell+k}{k}\in \Q[\ell]\middle| 0\le k\le n\right\}$, i.e., rational polynomials that map an integer to an integer. 
They are totally ordered by $p\le q$ if $p(\ell)\le q(\ell)$ for all $\ell$ sufficiently large.
Then $p\le q$ if and only if either $p=q$ or there is a $j\in\{0,\ldots,d-1\}$ such that $a_k=b_k$ for all $k>j$ and $a_j<b_j$, where $p=\sum_{k=0}^da_k\binom{\ell+k}{k}$ and $q=\sum_{k=0}^db_k\binom{\ell+k}{k}$.

\section{Proofs}\label{section_proof_dimension_polynomial}

\subsection{Existence Proof}\label{subsection_existence_proof}

Lemma~\ref{lemma_two_ways_to_get_algebraic_ideals} implies $I_{\le \ell}=\mcI(S_{\le \ell})$ and Theorem~\ref{theorem:krull_counting} states that $\dim(F\{U\}_{\le \ell}/I_{\le \ell})$ can be read off the number of polynomials in $S_{\le \ell}$, which only depends on $\ld(S)$.
Thus, to prove Theorem~\ref{theorem_general_dimension_polynomial}.\eqref{theorem_general_dimension_polynomial_exists} and \eqref{theorem_general_dimension_polynomial_form} we may assume $S=\ld(S)$.
In this case $\mcI(S)$ is a prime differential ideal, and hence the statements follow from Kolchin's original theorem \cite[\S II.12]{Kol}.

For the proof of \eqref{theorem_general_dimension_polynomial_differential_dimension} note that the transcendence basis of all associated primes of $\mcI(S)$ are equal by Theorem~\ref{theorem:krull_counting}, and for each of these associated prime the claim follows from Kolchin's original theorem.

To prove \eqref{theorem_general_dimension_polynomial_two_ideals} note that $I\subseteq J$ implies $I_{\le \ell}\subseteq J_{\le \ell}$ for all $\ell\ge 0$.
In particular, the map from $F\{U\}_{\le \ell}/I_{\le \ell}$ to $F\{U\}_{\le \ell}/J_{\le \ell}$ is surjective and, thus, $\dim(F\{U\}_{\le \ell}/I_{\le \ell})\ge \dim(F\{U\}_{\le \ell}/J_{\le \ell})$.

\subsection{Invariance Proof}

The differential polynomial ring $F\{U\}$ is filtered by the finitely generated $F$-algebras $F\{U\}_{\le \ell}$.
This filtration induces a filtration on $F\{U\}/I$ for a differential ideal $I$.
To prove the invariance statement in  Theorem~\ref{theorem_general_dimension_polynomial}.\eqref{theorem_general_dimension_polynomial_invariant} we show that this filtration extends to $\quot(F\{U\}/I)$ if $I$ is characterizable.
Thereby, standard techniques of filtrations can be adapted from Kolchin's proof.

\begin{example}
  Consider $\Delta=\{\partial_t\}$, $U=\{u,v\}$, and $I:=\langle u_0\cdot v_1\rangle_\Delta$.
  Then $u_0$ is no zero-divisor in $F\{U\}_{\le 0}/I_{\le 0}\cong F[u_0,v_0]$, but $u_0\cdot v_1=0$ in $F\{U\}/I$.
  So, the canonical map $\quot(F\{U\}_{\le 0}/I_{\le 0})\to \quot(F\{U\}_{\le 1}/I_{\le 1})=\quot(F[u_0,v_0,u_1,v_1]/\langle u_0\cdot v_1\rangle)$ is no inclusion, as $u_0^{-1}$ maps to zero.
\end{example}

\begin{lemma}\label{lemma_total_quotient_rings_filtered}
  Let $I\subseteq F\{U\}$ be a characterizable differential ideal and $\ell\in\Z_{\ge 0}$.
  Then, $F\{U\}_{\le \ell}/I_{\le \ell}\hookrightarrow F\{U\}_{\le \ell+1}/I_{\le \ell+1}$ induces an inclusion
  \begin{align*}
    \quot(F\{U\}_{\le \ell}/I_{\le \ell}) \hookrightarrow \quot(F\{U\}_{\le \ell+1}/I_{\le \ell+1})\mbox{ .}
  \end{align*}
\end{lemma}
\begin{proof}
  Any non-zero-divisor in $F\{U\}_{\le \ell}/I_{\le \ell}$ is a non-zero-divisor when considered in $F\{U\}_{\le \ell+1}/I_{\le \ell+1}$ (cf.\ Proposition~\ref{proposition_general_non_zerodivisor}), and thus a unit in $\quot(F\{U\}_{\le \ell+1}/I_{\le \ell+1})$.
  Hence, $F\{U\}_{\le \ell}/I_{\le \ell}\to \quot(F\{U\}_{\le \ell+1}/I_{\le \ell+1})$ factors over $\quot(F\{U\}_{\le \ell}/I_{\le \ell})$ by the universal property of localizations. 
  This induces a map $\iota:\quot(F\{U\}_{\le \ell}/I_{\le \ell})\to \quot(F\{U\}_{\le \ell+1}/I_{\le \ell+1})$.
  Now, $\ker\iota\cap F\{U\}_{\le \ell}/I_{\le \ell}$ is zero, since it is the kernel of the composition $F\{U\}_{\le \ell}/I_{\le \ell} \hookrightarrow F\{U\}_{\le \ell+1}/I_{\le \ell+1} \hookrightarrow \quot(F\{U\}_{\le \ell+1}/I_{\le \ell+1})$ of monomorphisms.
  The bijection between ideals in $F\{U\}_{\le \ell}/I_{\le \ell}$ not containing zero-divisors and all ideals in $\quot(F\{U\}_{\le \ell}/I_{\le \ell})$ implies $\ker\iota=0$.
\end{proof}

This filtration is well-behaved under differential birational maps.

\begin{lemma}\label{lemma_compatible_filtrations}
  Let $I\subseteq F\{U\}$ and $J\subseteq F\{V\}$ be characterizable differential ideals.
  Let $\varphi: \quot(F\{U\}/I) \to \quot(F\{V\}/J)$ be a differential birational map.
  Then there exists an $\ell_0\in\Z_{\ge 0}$ such that
  \begin{align*}
    \varphi(\quot(F\{U\}_{\le \ell}/I_{\le \ell})) &\subseteq \quot(F\{V\}_{\le \ell+\ell_0}/J_{\le \ell+\ell_0})
  \end{align*}
\end{lemma}
\begin{proof}
  $F\{U\}/I$ is a (left) $F[\Delta]$-module for every differential ideal $I\subset F\{U\}$, where $F[\Delta]$ is the ring of linear differential operators with coefficients in $F$.
  The filtration of $F[\Delta]$ by the linear differential operators $F[\Delta]_{\le k}$ of order $\le k$ is compatible with the filtration of $F\{U\}$ in the sense that $F[\Delta]_{\le k}(F\{U\}_{\le\ell}/I_{\le\ell})=F\{U\}_{\le\ell+k}/I_{\le\ell+k}$.
  
  There exists an $\ell_0\in\Z_{\ge0}$ with $\varphi(F\{U\}/I_{\le 0}) \subseteq \quot(F\{V\}/J_{\le \ell_0})$, as $F\{V\}/J=\bigcup_{\ell\in\Z_{\ge0}}F\{V\}_{\le\ell}/J_{\le\ell}$.
  Now
  \begin{align*}
    \varphi(\quot(F\{U\}/I_{\le \ell})) 
      &=         \varphi(\quot(F[\Delta]_{\le \ell}(F\{U\}/I_{\le 0})))\\
      &=         \quot(F[\Delta]_{\le \ell}\varphi(F\{U\}/I_{\le 0}))\\
      &\subseteq \quot(F[\Delta]_{\le \ell}\quot(F\{V\}/J_{\le \ell_0}))\\
      &\subseteq \quot(F\{V\}/J_{\le \ell+\ell_0})
  \end{align*}
\end{proof}

The Krull-dimension changes when passing to total quotient rings.
Instead, we use $\dim_F(R):=\max_{P\in\Ass(R)}\trdeg_F(\quot(R/P))$ as notion of dimension for $F$-algebras $R$.
Then, $\dim(R)=\dim_F(R)=\dim_F(\quot(R))$ allows to prove the invariance condition.

\begin{proof}[Proof of Theorem~\ref{theorem_general_dimension_polynomial}.\eqref{theorem_general_dimension_polynomial_invariant}]\label{theorem_general_dimension_polynomial_invariant_proof}
  Let $\varphi$ be as in Lemma~\ref{lemma_compatible_filtrations}.
  Then,
  \begin{align*}
    \quot(F\{U\}_{\le \ell}/I_{\le \ell}) &\cong \varphi(\quot(F\{U\}_{\le \ell}/I_{\le \ell})) \subseteq \quot(F\{V\}_{\le \ell+\ell_0}/J_{\le \ell+\ell_0})
  \end{align*}
  with the $\ell_0\in\Z_{\ge0}$ from Lemma~\ref{lemma_compatible_filtrations}, and thus
  \begin{align*}
    \dim(F\{U\}_{\le \ell}/I_{\le \ell})
      &= \dim_F(\quot(F\{U\}_{\le \ell}/I_{\le \ell}))\\
      &\le \dim_F(\quot(F\{V\}_{\le \ell+\ell_0}/J_{\le \ell+\ell_0}))\\
      &= \dim(F\{V\}_{\le \ell+\ell_0}/J_{\le \ell+\ell_0})\mbox{ .}
  \end{align*}
  Thus $\omega_I(\ell)\le \omega_J(\ell+\ell_0)$ and by symmetry $\omega_J(\ell)\le \omega_I(\ell+\ell_0)$.
  Now, an elementary argument implies that the degrees and leading coefficients of $\omega_I$ and $\omega_J$ agree.
\end{proof}

\subsection{Comparison Proof}

The proof of Theorem~\ref{theorem_general_dimension_polynomial}.\eqref{theorem_general_dimension_polynomial_leaders} and \eqref{theorem_general_dimension_polynomial_equality} uses two propositions, which relate ideals and squarefree regular chains.
The first proposition is a direct corollary to Lazard's Lemma (Theorem~\ref{theorem:krull_counting}).

\begin{proposition}\label{proposition_same_trancendece_basis}
  Let $S,S'$ be squarefree regular chains in $F[y_1,\ldots,y_n]$ with $\mcI(S)\subseteq\mcI(S')$ and $\left|S\right|=\left|S'\right|$.
  Then $\ld(S)=\ld(S')$.
\end{proposition}

The following lemma is used to prove the second proposition.
It captures an obvious property of the pseudo reduction with respect to a squarefree regular chain.

\begin{lemma}\label{lemma_pseudo_reduction}
  Let $S$ be a squarefree regular chain and $p\in F[y_1,\ldots,y_n]$ with $\ld(p)=x$, $p\in\mcI(S)$, and $\ini(p)\not\in\mcI(S)$.
  Then $S$ has an element of leader $x$ and $\deg_x(S_x)\le\deg_x(p)$.
\end{lemma}

\begin{proposition}\label{proposition_same_algebraic_ideals}
  Let $S$ and $S'$ be squarefree regular chains in $R=F[y_1,\ldots,y_n]$ with $\mcI(S)\subseteq\mcI(S')$ and $\left|S\right|=\left|S'\right|$.
  Then, $\mcI(S)=\mcI(S')$ if and only if $\deg_x(S_x)=\deg_x(S_x')$ for all $x\in\ld(S)=\ld((S'))$.
\end{proposition}

\begin{proof}
  Let $\deg_x(S_x)=\deg_x(S_x')$ for all $x\in\ld(S)$.
  We show $\mcI(S)\supseteq\mcI(S')$ by a Noetherian induction.
  The statement is clear for the principle ideals $\mcI(S_{<y_2})$ and $\mcI(S_{<y_2}')$.
  Let $p\in \mcI(S')$ with $\ld(p)=y_i$ and $\deg_{y_i}(p)=j$.
  Assume by induction that $q\in\mcI(S')$ implies $q\in\mcI(S)$ for all $q$ with $\ld(q)<y_i$ or $\ld(q)=y_i$ and $\deg_{y_i}(q)<j$.
  Without loss of generality $\ini(p)\not\in\mcI(S')_{<y_i}=\mcI(S)_{<y_i}$, as otherwise $p$ has a lower degree in $y_i$ or a lower ranking leader when substituting $\ini(p)$ by zero.
  Now, Lemma~\ref{lemma_pseudo_reduction} implies $y_i\in\ld(S)$ and $\deg_{y_i}(p)\ge\deg_{y_i}(S_{y_i})$.
  Then,
  \[
    r:=\ini(S_{y_i})\cdot p-\ini(p)\cdot y_i^{\deg_{y_i}(p)-\deg_{y_i}(S_{y_i})}\cdot S_{y_i}
  \]
  is in $\mcI(S)$ if and only if $p\in \mcI(S)$ is, but $r$ is of lower degree or of lower ranking leader than $p$.
  The claim follows by induction.
  
  Let $\mcI(S)=\mcI(S')$ and $x\in\ld(S)$.
  This implies $\ini(S_x)\not\in\mcI(S')$, and thus $\deg_x(S'_x)\le\deg_x(S_x)$ by Lemma~\ref{lemma_pseudo_reduction}.
  By symmetry $\deg_x(S'_x)\ge\deg_x(S_x)$, and thus $\deg_x(S_x)=\deg_x(S_x')$.
\end{proof}

\begin{proof}[Proof of Theorem~\ref{theorem_general_dimension_polynomial}.\eqref{theorem_general_dimension_polynomial_equality}]\label{theorem_general_dimension_polynomial_equality_proof}
  Lemma~\ref{lemma_two_ways_to_get_algebraic_ideals} reduces the statements to the algebraic case.
  In this case, Proposition~\ref{proposition_same_trancendece_basis} implies \eqref{theorem_general_dimension_polynomial_leaders}, and \eqref{theorem_general_dimension_polynomial_equality} follows from Proposition~\ref{proposition_same_algebraic_ideals}, because all polynomials in $S_{\le \ell}$ of degree greater than one in their respective leader already lie in $S$, $\ell\in\Z_{\ge 0}$.
\end{proof}

\section{Computation of the Differential Dimension Polynomial}\label{section_computation}

To compute the differential dimension polynomial $\omega_{\mcI(S)}$ of a characterizable differential ideal $\mcI(S)\subseteq F\{U\}$ for a differential regular chain $S$ we may assume $S=\ld(S)$ (cf.\ Subsection~\ref{subsection_existence_proof}).
This assumption implies that $\mcI(S)$ is a prime differential ideal, and for this case there exist well-known combinatorial algorithms for $\omega_{\mcI(S)}$ \cite{DifferentialDimensionPolynomials}.

Alternatively, the differential dimension polynomial $\omega_{\mcI(S)}$ can be read off the set of equations $S$ of a simple differential system \cite{thomasalg_jsc}.
Such a set $S$ is almost a differential regular chain, except that weak triangularity is replaced by the Janet decomposition, which associates a subset of $\Delta$ of cardinality $\zeta_p$ to each $p\in S$.
Then, the differential dimension polynomial is given by the closed formula
\begin{align*}
  \omega_{\mcI(S)}(l)=m\binom{n+\ell}{n}-\sum_{p\in S}\binom{\zeta_p+\ell-\ord(\ld(p))}{\zeta_p}\mbox{ ,}
\end{align*}
involving only the cardinalities $\zeta_p$ and the orders $\ord(\ld(p))$.

\section{Examples}\label{section_examples}

For each prime differential ideal $I$ there exists a differential regular chain $S$ with $I=\mcI(S)$.
Thus, the differential dimension polynomial defined in Theorem~\ref{theorem_general_dimension_polynomial} includes the version of Kolchin.
However, the following example shows that Theorem~\ref{theorem_general_dimension_polynomial} is more general.

\begin{example}\label{example_dimpoly_generalize}
  Consider $U=\{u,v\}$, $\Delta=\{\partial_t\}$, $p=u_1^2-v$, and $q=v_1^2-v$.
  The characterizable differential ideal $I:=\mcI(\{p,q\})$ is not prime, as $p-q=(u_1-v_1)(u_1+v_1)$.
\end{example}

Prime differential ideals $I\subseteq J$ are equal if and only if $\omega_I=\omega_J$ by Kolchin's theorem.
By the following example, this is wrong for characterizable ideals and any generalization to such ideals needs to consider the degrees of polynomials in a differential regular chain.

\begin{example}
  Consider $\langle u_0^2-u_0\rangle_\Delta=\mcI(\{u_0^2-u_0\})\subsetneq\langle u_0\rangle_\Delta=\mcI(\{u_0\}$ in $F\{u\}$ for $|\Delta|=1$.
  Both differential ideals are characterizable and have the differential dimension polynomial $0$.
  However, they are not equal.
\end{example}

The next example shows that the Cartan characters and other invariants do not suffice to prove that two solution sets are not equal.

\begin{example}\label{example_cartan_weak}
  For $\Delta=\{\partial_x,\partial_y\}$ consider the regular chains $S_1=\{u_{1,0}\}$ and $S_2=\{u_{2,0},u_{1,1}\}$ in $\C\{u\}$.
  Then $\mcI(S_2)\subseteq\mcI(S_1)$.
  The strength and first Cartan character are one and the second Cartan character and differential dimension are zero for both ideals (in any order high enough), i.e., these values agree for both ideals.
  However, $\mcI(S_2)\not=\mcI(S_1)$, as $\omega_{\mcI(S_1)}(\ell)=l+1\not=l+2=\omega_{\mcI(S_2)}(\ell)$.
\end{example}

In the last example, the differential dimension polynomial proves that a symbolic differential equation solver does not find all solutions.

\begin{example}\label{example_heat_burgers}
  Let $U=\{u\}$ and $\Delta=\{\frac{\partial}{\partial t},\frac{\partial}{\partial x}\}$.
  The viscous {\sc Burgers}' equation $b=u_{0,2}-u_{1,0}-2u_{0,1}\cdot u_{0,0}$ has differential dimension polynomial is $2\ell+1$.
  \textsf{MAPLE}'s \textsf{pdsolve} \cite{maple17} finds the set
  \begin{align*} 
    T:=\Big\{c_1\tanh(c_1x+c_2t+c_3)-\frac{c_2}{2c_1}\Big|c_1, c_2, c_3\in\C, c_1\not=0\Big\}
  \end{align*}
  of solutions, which only depends on three parameters.
  The differential dimension polynomial shows that the set of solutions is infinite dimensional, and hence $T$ is only a small subset of all solutions.
\end{example}

\bibliographystyle{plain}

\begin{thebibliography}{10}

\bibitem{maple17}
Maple 17.00.
\newblock Maplesoft, a division of Waterloo Maple Inc., Waterloo, Ontario.

\bibitem{thomasalg_jsc}
T.~B{\"a}chler, V.~P. Gerdt, M.~Lange-Hegermann, and D.~Robertz.
\newblock Algorithmic {T}homas decomposition of algebraic and differential
  systems.
\newblock {\em J. Symbolic Comput.}, 47(10):1233--1266, 2012.
\newblock (\href{http://arxiv.org/abs/1108.0817}{\texttt{arXiv:1108.0817}}).

\bibitem{BLOP}
F.~Boulier, D.~Lazard, F.~Ollivier, and M.~Petitot.
\newblock Computing representations for radicals of finitely generated
  differential ideals.
\newblock {\em Appl. Algebra Engrg. Comm. Comput.}, 20(1):73--121, 2009.

\bibitem{Cartan1931}
{\'E}.~Cartan.
\newblock Sur la th\'eorie des syst\`emes en involution et ses applications \`a
  la relativit\'e.
\newblock {\em Bulletin de la Soci\'et\'e Math\'ematique de France},
  59:88--118, 1931.

\bibitem{ChebTerrabRoche2008}
E.~S. Cheb-Terrab and A.~D. Roche.
\newblock Hypergeometric solutions for third order linear odes, 2008.
\newblock \href{http://arxiv.org/abs/0803.3474}{(\texttt{arXiv:0803.3474})}.

\bibitem{ChebTerrab1995102}
E.~S. Cheb-Terrab and K.~von B{\"u}low.
\newblock A computational approach for the analytical solving of partial
  differential equations.
\newblock {\em Computer Physics Communications}, 90(1):102 -- 116, 1995.

\bibitem{EinsteinRelativityAppenix2Supplement}
A.~Einstein.
\newblock {\em Supplement to {A}ppendix {II} of ``{T}he meaning of relativity,
  4th ed.''}.
\newblock Princeton University Press, N. J., 1953.

\bibitem{GaoChowForm}
Xiao-Shan Gao, Wei Li, and Chun-Ming Yuan.
\newblock Intersection theory in differential algebraic geometry: generic
  intersections and the differential {C}how form.
\newblock {\em Trans. Amer. Math. Soc.}, 365(9):4575--4632, 2013.

\bibitem{Hubert1}
E.~Hubert.
\newblock Notes on triangular sets and triangulation-decomposition algorithms.
  {I}. {P}olynomial systems.
\newblock In {\em Symbolic and numerical scientific computation ({H}agenberg,
  2001)}, volume 2630 of {\em Lecture Notes in Comput. Sci.}, pages 1--39.
  Springer, Berlin, 2003.

\bibitem{JohnsonPolynomial}
J.~Johnson.
\newblock Differential dimension polynomials and a fundamental theorem on
  differential modules.
\newblock {\em Amer. J. Math.}, 91:239--248, 1969.

\bibitem{KolPolynomial}
E.~R. Kolchin.
\newblock The notion of dimension in the theory of algebraic differential
  equations.
\newblock {\em Bull. Amer. Math. Soc.}, 70:570--573, 1964.

\bibitem{Kol}
E.~R. Kolchin.
\newblock {\em Differential algebra and algebraic groups}.
\newblock Academic Press, New York, 1973.
\newblock {Pure and Applied Mathematics, Vol. 54}.

\bibitem{DifferentialDimensionPolynomials}
M.~V. Kondratieva, A.~B. Levin, A.~V. Mikhalev, and E.~V. Pankratiev.
\newblock {\em Differential and difference dimension polynomials}, volume 461
  of {\em Mathematics and its Applications}.
\newblock Kluwer Academic Publishers, Dordrecht, 1999.

\bibitem{lhphd}
M.~Lange-Hegermann.
\newblock {\em {Counting Solutions of Differential Equations}}.
\newblock PhD thesis, RWTH Aachen, Germany, 2014.

\bibitem{LevinDimensionPolynomialIntermediateFields}
A.~Levin.
\newblock Dimension polynomials of intermediate fields and {K}rull-type
  dimension of finitely generated differential field extensions.
\newblock {\em Math. Comput. Sci.}, 4(2-3):143--150, 2010.

\bibitem{Rosenfeld}
A.~Rosenfeld.
\newblock Specializations in differential algebra.
\newblock {\em Trans. Amer. Math. Soc.}, 90:394--407, 1959.

\end{thebibliography}

\end{document}